\documentclass[12pt,a4paper]{article}
\usepackage{amsmath,amssymb,amsbsy,amscd,amsthm}
\theoremstyle{plain}
\newtheorem{thm}{Theorem}[section]
\newtheorem{lem}[thm]{Lemma}

\newtheorem{cor}[thm]{Corollary}
\newtheorem{defn}[thm]{Definition}

\newcommand{\ch }{\mathop{\rm char}\nolimits}
\newcommand{\id }{{\rm id}}
\newcommand{\nd }{\noindent}
\newcommand{\degw }{\mathop{\rm deg}_{\w }\nolimits}

\newcommand{\Aut }{\mathop{\rm Aut}\nolimits}
\newcommand{\Aff }{\mathop{\rm Aff}\nolimits}
\newcommand{\T }{\mathop{\rm T}\nolimits}
\newcommand{\E }{\mathop{\rm E}\nolimits}
\newcommand{\GA }{\mathop{\rm GA}\nolimits}
\newcommand{\GL }{\mathop{\it GL}\nolimits}
\newcommand{\ngg }{\mathop{\rm NG}\nolimits}

\newcommand{\zs}{\{ 0\} }
\newcommand{\sm}{\setminus}
\newcommand{\F}{{\bf F}}

\newcommand{\Z}{{\bf Z}}
\newcommand{\Q}{{\bf Q}}

\newcommand{\x}{{\bf x}}
\newcommand{\w}{{\bf w}}

\newcommand{\Rx}{R[{\bf x}]}
\newcommand{\Rxh}{R[\hat{\bf x}]}
\newcommand{\ep}{\epsilon}

\begin{document}

\title%[Stably co-tame polynomial automorphisms]
{Stably co-tame polynomial automorphisms 
over commutative rings}

\author{Shigeru Kuroda\thanks{
Partly supported by JSPS KAKENHI 
Grant Number 15K04826}}

\date{}

\maketitle

\begin{abstract}
We say that a polynomial automorphism $\phi $ 
in $n$ variables is {\it stably co-tame} 
if the tame subgroup in $n$ variables 
is contained in the subgroup generated by $\phi $ 
and affine automorphisms in $n+1$ variables. 
In this paper, 
we give conditions for stably co-tameness 
of polynomial automorphisms. 
\end{abstract}

\section{Introduction}
\setcounter{equation}{0}

Let $R$ be a commutative ring of characteristic $p\ne 1$, 
$\Rx :=R[x_1,\ldots ,x_n]$ 
the polynomial ring in $n$ variables over $R$, 
and $\GA _n(R):=\Aut _R\Rx $ 
the automorphism group of the $R$-algebra $\Rx $. 
We identify each $\phi \in \GA _n(R)$ 
with the $n$-tuple 
$(\phi (x_1),\ldots ,\phi (x_n))$ of elements of $\Rx $. 
The composition is defined by 
\begin{equation}\label{eq:composition}
\phi \circ \psi 
=(g_1(f_1,\ldots ,f_n),\ldots ,g_n(f_1,\ldots ,f_n))
\end{equation}
for $\phi =(f_1,\ldots ,f_n),\psi =(g_1,\ldots ,g_n)\in \Rx ^n$. 
For each $r\ge 1$, 
we regard $\GA _n(R)$ as a subgroup of $\GA _{n+r}(R)$ 
by identifying each $\phi \in \GA _n(R)$ 
with the unique extension $\tilde{\phi }\in \GA _{n+r}(R)$ 
of $\phi $ defined by $\tilde{\phi }(x_{n+i})=x_{n+i}$ 
for $i=1,\ldots ,r$.

We say that $\phi \in \GA _n(R)$ is {\it affine} 
if $\phi =(x_1,\ldots ,x_n)A+b$ 
for some $A\in \GL _n(R)$ and $b\in R^n$, 
and set $\Aff _n(R):=\{ \phi \in \GA _n(R)\mid 
\phi \ {\rm is\ affine}\} $. 
We define 
$$
\ep (f):=(x_1+f,x_2,\ldots ,x_n)
\in \GA _n(R)\ \ {\rm for\ each}\ \ 
f\in \Rxh :=R[x_2,\ldots ,x_n], 
$$
and set $\E _n(R):=\{ \ep (f)\mid f\in \Rxh \} $. 
We call $\T _n(R):=\langle \Aff _n(R), \E _n(R)\rangle $ 
the {\it tame subgroup}, 
and elements of $\T _n(R)$ are said to be {\it tame}. 
Here, for subsets $S_1,\ldots ,S_r$ 
and elements $g_1,\ldots ,g_s$ of a group $G$, 
we denote by $\langle S_1,\ldots ,S_r,g_1,\ldots ,g_s\rangle $ 
the subgroup of $G$ generated by 
$\bigcup _{i=1}^rS_i\cup \{ g_1,\ldots ,g_s\} $.

If $n\ge 3$ and $R$ contains $\Q $, 
then $\T _n(R)=\langle \Aff _n(R),\sigma \rangle $ 
holds for $\sigma =\ep (x_2^2)$ by Derksen 
(cf.~\cite[Thm.\ 5.2.1]{Essen}). 
We remark that Derksen's theorem 
requires that 
$R$ is generated by $R^*$ as a $\Q $-vector space, 
but this assumption is in fact unnecessary 
(cf.~\S \ref{sect:main1}). 
When $R$ is a field and $p=0$, 
Bodnarchuk~\cite{B2} proved 
a similar result 
for more general $\sigma $'s. 
The situation is  different if $p$ is a prime. 
Maubach-Willems~\cite{MW} 
showed that 
$\T _3(\F_2)\ne \langle \Aff _3(\F _2),\ep (x_2^2) \rangle $, 
and conjectured that, 
if $n\ge 3$ and $R$ is a finite field, 
then no finite subset $E$ of $\T _n(R)$ 
satisfies $\T _n(R)=\langle \Aff _n(R),E\rangle $.

Edo~\cite{Edo cotame} 
found a class of 
$\phi \in \GA _n(R)$ 
for which $\langle \Aff _n(R),\phi \rangle $ 
contains $\T _n(R)$. 
Such $\phi $ is said to be {\it co-tame}. 
If $R$ is a field, 
no element of $\GA _2(R)$ is co-tame thanks to 
Jung~\cite{Jung} and van der Kulk~\cite{Kulk}. 
For $n\ge 3$, 
it is easy to find elements of $\GA _n(R)\sm \Aff _n(R)$ 
which are not co-tame 
if $R$ is not a field 
(cf.~(\ref{eq:reduction})), 
or if $p$ is a prime (cf.~\cite[\S 1 (d)]{EK}). 
In the case where $n\ge 3$, $R$ is a field and $p=0$, 
the first example of such an automorphism was 
found by Edo-Lewis~\cite{EL} for $n=3$.

Recall that $\phi \in \GA _n(R)$ 
is said to be {\it stably tame} 
if $\phi $ belongs to $\T _{n+1}(R)$. 
It is known that some non-tame automorphisms 
are stably tame 
(cf.~\cite{BEW}, \cite{Nagata}, \cite{SU}, \cite{Smith}). 
The following is an analogue 
to the stably tame automorphisms.

\begin{defn}\rm 
We say that 
$\phi \in \GA _n(R)$ is {\it stably co-tame} 
if $\langle \Aff _{n+1}(R),\phi \rangle $ 
contains $\T _n(R)$, 
or equivalently 
$\langle \Aff _{n+1}(R),\phi \rangle $ 
contains $\E _n(R)$. 
\end{defn}

Clearly, 
co-tame automorphisms are stably co-tame. 
When $R$ is a field, 
there exist elements of $\T _3(R)$ 
which are not co-tame but stably co-tame 
in both cases $p=0$ and $p>0$ 
(cf.~\S \ref{sect:rmk}.1). 
The purpose of this paper is to study when 
elements of $\GA _n(R)$ are 
stably co-tame or not. 
If $R$ contains an infinite field, 
we have a necessary 
and sufficient condition for stably co-tameness 
(Corollary~\ref{cor:main}).

This paper is organized as follows. 
The main results are stated in Section~\ref{sect:main}, 
and three key results are proved in 
Sections~\ref{sect:main1}, \ref{sect:IP} 
and \ref{sect:main2}. 
In Section~\ref{sect:rmk}, 
we study stably co-tameness of the example of Edo-Lewis. 
We also discuss a technique which is useful 
when $R$ does not contain an infinite field.

\section{Main results}\label{sect:main}
\setcounter{equation}{0}

Since $\T _1(R)=\Aff _1(R)$, 
we always assume that $n\ge 2$ unless otherwise stated. 
Take any $\phi \in \GA _n(R)$. 
We define $M_{\phi }$ to be the $R$-submodule of $\Rx $ 
generated by 1, $x_i$ and $\eta (\phi (x_i))$ 
for $i=1,\ldots ,n$ and $\eta \in \Aff _n(R)$. 
By definition, 
we have 
\begin{equation}\label{eq:M phi rmk}
\eta (M_{\phi })\subset M_{\phi }\quad
{\rm for\ each}\quad 
\eta\in \Aff _n(R). 
\end{equation}

The following theorem 
holds for any commutative ring $R$ 
of characteristic $p\ne 1$.

\begin{thm}\label{thm:main1}
$\phi \in \GA _n(R)$ is stably co-tame 
in the following four cases:

\nd {\rm (a)} 
$M_{\phi }$ contains $x_ix_j$ for some $1\le i<j\le n$.

\nd {\rm (b)}
$2$ is a unit of $R$, 
and $M_{\phi }$ contains $x_i^2$ 
for some $1\le i\le n$.

\nd {\rm (c)}
$n=p=2$ and $M_{\phi }$ contains $x_1^2x_2$ or $x_1x_2^2$.

\nd {\rm (d)}
$n=p=2$, 
$M_{\phi }$ contains $x_1^3$ or $x_2^3$, 
and there exists $\xi \in R^*$ 
satisfying $\xi +1\in R^*$. 
\end{thm}

Next, 
assume that $p$ is zero or a prime. 
For $a\in R$ and $t_1,\ldots ,t_n\ge 0$, 
we call $ax_1^{t_1}\cdots x_n^{t_n}$ 
a {\it good monomial} in the following five cases:

\smallskip

\nd (I) 
$p=0$ and $t_1+\cdots +t_n\ge 2$.

\nd (II) 
$p\ge 2$ and 
there exist $1\le i<j\le n$ 
such that $t_i,t_j\equiv 1\pmod{p}$.

\nd (III) 
$p\ge 3$ and 
there exists $1\le i\le n$ such that $t_i\not\equiv 0,1\pmod{p}$.

\nd (IV) 
$n=p=2$ and there exist $i,j\in \{ 1,2\} $ such that 
$t_i\equiv 1$ and 
$t_j\equiv 2\pmod{4}$.

\nd (V) 
$n=p=2$ and there exists $1\le i\le 2$ such that 
$t_i\equiv 3\pmod{4}$.

\smallskip

For each $f\in \Rx $, 
let $C_f$ denote the set of the coefficients of good monomials 
appearing in $f$. 
We define $I_{\phi }$ to be the ideal of $R$ 
generated by $\bigcup _{i=1}^nC_{\phi (x_i)}$.

\begin{thm}\label{thm:NG}
Assume that $p$ is zero or a prime. 
If $\phi \in \GA _n(R)$ satisfies $I_{\phi }\ne R$, 
then $\phi $ is not stably co-tame. 
\end{thm}

Throughout this paper, 
let $k$ be a field. 
When $R$ is a commutative $k$-algebra, 
we say that $f\in \Rx $ satisfies the {\it degree condition} if 
\begin{equation}\label{eq:degree condition}
\deg _{x_i}^sf\le \# k-2
\quad {\rm for}\quad 
i=1,\ldots ,n. 
\end{equation}
Here, 
$\deg _{x_i}^sf$ denotes 
the separable degree of $f$ if $p>0$, 
and the standard degree $\deg _{x_i}f$ of $f$ if $p=0$, 
as a polynomial in $x_i$. 
The {\it separable degree} of 
a polynomial $g(x)\in R[x]$ in one variable 
is defined as the degree of $h(x)\in R[x]\sm R[x^p]$ such that 
$h(x^{p^e})=g(x)$ for some $e\ge 0$. 
We say that $\phi \in \GA _n(R)$ 
satisfies the {\it degree condition} 
if $\phi (x_i)$ satisfies the degree condition 
for $i=1,\ldots ,n$.

Now, we define $J_{\phi }$ 
to be the ideal of $R$ generated by the union of $C_f$ 
for $f\in \sum _{i=1}^nR\phi (x_i)+\sum _{i=1}^nRx_i$ 
satisfying the degree condition. 
Since no good monomial is linear, 
$J_{\phi }$ is contained in $I_{\phi }$. 
If $\phi $ 
satisfies the degree condition, 
then $J_{\phi }$ is equal to $I_{\phi }$. 
Hence, 
(ii) of the following theorem is a consequence of (i).

\begin{thm}\label{thm:main2}
Let $k$ be a field, 
$R$ a commutative $k$-algebra, 
and $\phi \in \GA _n(R)$.

\noindent{\rm (i)} 
If $J_{\phi }=R$, 
then $\phi $ is stably co-tame.

\noindent{\rm (ii)} 
If $\phi $ satisfies the degree condition 
and if $I_{\phi }=R$, 
then $\phi $ is stably co-tame. 
\end{thm}

From Theorems~\ref{thm:NG} and \ref{thm:main2} (ii), 
we obtain the following corollary.

\begin{cor}\label{cor:main}
Let $k$ be an infinite field, 
and $R$ a commutative $k$-algebra. 
Then, 
$\phi \in \GA _n(R)$ 
is stably co-tame if and only if $I_{\phi }=R$. 
\end{cor}

In particular, 
if $R=k$ is an infinite field, 
then 
$\phi \in \GA _n(k)$ is stably co-tame 
if and only if a good monomial appears in 
$\phi (x_i)$ for some $i$. 
When $p=0$, 
this is the same as $\phi $ is non-affine.

In the following three sections, 
we prove Theorems~\ref{thm:main1}, 
\ref{thm:NG} and \ref{thm:main2} (i).

\section{Proof of Theorem~\ref{thm:main1}}\label{sect:main1}
\setcounter{equation}{0}

Let $R$ be any commutative ring, 
and $\Gamma $ a subgroup of $\GA _{n+1}(R)$ 
containing $\Aff _{n+1}(R)$, where $n\ge 2$. 
First, we study properties of $\Gamma $. 
Define 
$$
\hat{\ep }(f):=(x_1,\ldots ,x_n,x_{n+1}+f)
\in \GA _{n+1}(R)\ \ {\rm for\ each}\ \ f\in \Rx . 
$$
We identify each permutation $\sigma \in S_{n+1}$ with 
$(x_{\sigma (1)},\ldots ,x_{\sigma (n+1)})\in \Aff _{n+1}(R)$, 
and write $\phi ^{\sigma }:=\sigma ^{-1}\circ \phi \circ \sigma $ 
for $\phi \in \GA _{n+1}(R)$. 
Then, 
we have the following:

\nd (A) 
If $\phi =(f_1,\ldots ,f_n)\in \GA _n(R)$ 
belongs to $\Gamma $, 
then 
$\hat{\ep }(f_i)=\phi \circ \hat{\ep }(x_i)\circ \phi ^{-1}$ 
belongs to $\Gamma $ for each $1\le i\le n$.

\nd (B) If $f,g\in \Rxh $ 
satisfy $\ep (f),\ep (gx_{n+1})\in \Gamma $, 
then we have $\hat{\ep }(f)=\ep (f)^{(1,n+1)}\in \Gamma $, 
and so 
$$
\ep (fg)=\hat{\ep }(f)\circ \ep (gx_{n+1})\circ 
\hat{\ep }(f)^{-1}\circ \ep (gx_{n+1})^{-1}\in \Gamma . 
$$
In particular, 
$\ep (f)\in \Gamma $ implies 
$\ep (af)\in \Gamma $ for each $f\in \Rxh $ and $a\in R$, 
since $\ep (ax_{n+1})$ belongs to $\Aff _{n+1}(R)$, 
and hence to $\Gamma $. 

\medskip 

When $\Q \subset R$ and $n\ge 3$, 
Derksen showed that $\T _n(R)$ is generated by 
$\Aff _n(R)$ and $\{ \ep (ax_2^2)\mid a\in R\} $, 
and $\T _n(R)=\langle \Aff _n(R),\ep (x_2^2)\rangle $ holds 
if $R$ is generated by $R^*$ as a $\Q $-vector space 
(cf.~\cite[Thm.\ 5.2.1]{Essen}). 
We remark that, 
since $n\ge 3$, 
the first statement and (B) imply that 
$\T _n(R)=\langle \Aff _n(R),\ep (x_2^2)\rangle $ 
whenever $\Q \subset R$.

\begin{lem}\label{lem:Gamma}
If $\Gamma $ contains $\hat{\ep }(x_1x_2)$, 
or if $n=2$ and $\Gamma $ contains $\hat{\ep }(x_1x_2^2)$, 
then $\Gamma $ contains $\T_{n}(R)$. 
\end{lem}
\begin{proof}
We show that $\Gamma $ contains $\ep (f)$ 
for each $f\in \Rxh $. 
Since $\ep (g+h)=\ep (g)\circ \ep (h)$ 
holds for each $g,h\in \Rxh $, 
we may assume that $f$ is a monomial. 
In the former case, 
$\Gamma $ contains 
$\hat{\ep }(x_1x_2)^{(1,n+1)(2,i)}=\ep (x_ix_{n+1})$ 
for each $2\le i\le n$. 
By (B), 
it follows that 
$\ep (f)\in \Gamma $ implies $\ep (x_if)\in \Gamma $. 
Since 
$\Gamma $ contains $\ep (a)$ for each $a\in R$, 
the assertion follows by induction on $\deg f$. 
In the latter case, 
$\Gamma $ contains $\hat{\ep }(x_1x_2^2)^{(1,3)}=\ep (x_2^2x_3)$. 
Hence, 
$\ep (ax_2^l)\in \Gamma $ 
implies $\ep (ax_2^{l+2})\in \Gamma $ 
for each $a\in R$ and $l\ge 0$ similarly. 
Since $\Gamma $ contains $\ep (a)$ and $\ep (ax_2)$, 
it follows that $\Gamma $ contains 
$\ep (ax_2^l)$ for all $l\ge 0$. 
\end{proof}

The following two implications hold for the conditions listed in 
Theorem~\ref{thm:main1}.

\begin{lem}\label{lem:key1}
{\rm (b)} implies {\rm (a)}, 
and {\rm (d)} implies {\rm (c)}. 
\end{lem}
\begin{proof}
By (\ref{eq:M phi rmk}), 
(b) implies $(x_1+x_2)(x_1-x_2)=x_1^2-x_2^2\in M_{\phi }$. 
Since 2 is a unit of $R$, 
we have 
$(x_1+x_2,x_1-x_2,x_3,\ldots ,x_n)\in \Aff _n(R)$. 
Hence, 
$M_{\phi }$ contains $x_1x_2$ by (\ref{eq:M phi rmk}). 
Similarly, 
(d) implies 
$x_1^2x_2+x_1x_2^2=x_1^3+x_2^3+(x_1+x_2)^3\in M_{\phi }$. 
Hence, 
$M_{\phi }$ contains 
$((\xi x_1)^2x_2+\xi x_1x_2^2)+
\xi ^2(x_1^2x_2+x_1x_2^2)
=\xi (\xi +1)x_1x_2^2$. 
Since $\xi (\xi +1)$ is a unit of $R$, 
this implies that $M_{\phi }$ contains $x_1x_2^2$. 
\end{proof}

Now, 
we prove Theorem~\ref{thm:main1}. 
Thanks to Lemmas~\ref{lem:Gamma} and \ref{lem:key1} 
and (\ref{eq:M phi rmk}), 
it suffices to show that 
$\Gamma :=\langle \Aff _{n+1}(R),\phi \rangle $ 
contains $\hat{\ep }(f)$ for each $f\in M_{\phi }$. 
Write $\phi =(f_1,\ldots ,f_n)$ 
and 
$$
f=h+\sum _{i=1}^n\sum _{j=1}^ra_{i,j}\eta _j(f_i),\ {\rm where}\ 
h\in \sum _{i=1}^nRx_i+R,\ 
a_{i,j}\in R\ {\rm and}\ \eta _j\in \Aff _n(R). 
$$
Then, 
$\hat{\ep }(f)$ is the product of $\hat{\ep }(h)$ 
and $\hat{\ep }(\eta _j(a_{i,j}f_i))$ 
for $i=1,\ldots ,n$ and $j=1,\ldots ,r$. 
Since $\Gamma $ contains 
$\hat{\ep }(f_i)$ for each $i$ by (A), 
$\Gamma $ contains $\hat{\ep }(a_{i,j}f_i)$ 
for each $i$, $j$ by (B). 
Hence, 
$\Gamma $ contains 
$\eta _j\circ \hat{\ep }(a_{i,j}f_i)\circ \eta _j^{-1}
=\hat{\ep }(\eta _j(a_{i,j}f_i))$ 
for each $i$, $j$. 
Since $\hat{\ep }(h)$ is affine, 
it follows that $\Gamma $ contains $\hat{\ep }(f)$. 
This completes the proof of 
Theorem~\ref{thm:main1}.

\section{Proof of Theorem~\ref{thm:NG}}\label{sect:IP}
\setcounter{equation}{0}

Assume that $p$ is zero or a prime. 
We define $\ngg _n(R)$ to be the set of $\phi \in \GA _n(R)$ 
such that no good monomial appears in 
$\phi (x_1),\ldots ,\phi (x_n)$. 
If $p=0$, 
then we have $\ngg _n(R)=\Aff _n(R)$. 
In this case, 
the following theorem is obvious.

\begin{thm}\label{thm:ngg}
$\ngg _n(R)$ is a subgroup of $\GA _n(R)$, 
and no element of $\ngg _n(R)$ is stably co-tame. 
In fact, 
$\E _n(R)\not\subset \langle \Aff _{n+r}(R),
\ngg _{n}(R)\rangle $ 
holds for any $r\ge 1$. 
\end{thm}

To prove Theorem~\ref{thm:ngg}, 
we need a lemma, 
where $p$ need not be zero or a prime. 
Consider the standard grading 
$\Rx =\bigoplus _{l\geqslant 0}\Rx _l$. 
An $R$-submodule $V$ of $\Rx $ 
is said to be {\it graded} if $V=\bigoplus _{l\geqslant 0}(V\cap \Rx _l)$. 
If $V$ is generated by monomials, 
then $V$ is graded. 
Recall that each $(f_1,\ldots ,f_n)\in \Rx ^n$ is 
identified with the substitution map $\Rx \to \Rx $ 
defined by $x_i\mapsto f_i$ for $i=1,\ldots ,n$, 
and $\Rx ^n$ forms a monoid 
for the composition defined in (\ref{eq:composition}). 
Note that $V^n$ is closed under this operation 
if and only if $\phi (V)\subset V$ holds for each $\phi \in V^n$.

\begin{lem}\label{lem:IP}
Let $V$ be a graded $R$-submodule of $\Rx $ 
such that $V^n$ is closed under composition. 
If $x_1,\ldots ,x_n\in V$, 
then $V^n\cap \GA _n(R)$ is a subgroup of $\GA _n(R)$. 
\end{lem}
\begin{proof}
Since $V^n\cap \GA _n(R)$ contains $(x_1,\ldots ,x_n)$, 
and is closed under composition, 
we show that $\phi ^{-1}$ belongs to $V^n$ 
for each $\phi \in V^n\cap \GA _n(R)$. 
There exists $\eta \in V^n\cap \Aff _n(R)$ 
for which $\psi:=\phi \circ \eta $ satisfies 
$\psi(x_i)\in x_i+\bigoplus _{l\geqslant 2}\Rx _l$ 
for $i=1,\ldots ,n$. 
Since $\phi ^{-1}=\eta \circ \psi^{-1}$, 
and $V^n$ is closed under composition, 
it suffices to verify that 
$\psi^{-1}$ belongs to $V^n$. 
Suppose to the contrary that $f:=\psi^{-1}(x_i)$ 
does not belong to $V$ for some $i$. 
Write $f=\sum _{l\geqslant 0}f_l$, 
where $f_l\in \Rx _l$. 
Then, 
$f_d\not\in V$ holds for some $d\ge 0$. 
Take the minimal $d$, 
and set $f':=\sum _{l<d}f_l$ and $f'':=\sum _{l>d}f_l$. 
Let $h$, $h'$ and $h''$ be the homogeneous 
components of 
$\psi(f_d)$, $\psi(f')$ and $\psi(f'')$ of degree $d$, 
respectively. 
Then, 
$h+h'+h''$ belongs to $V$, 
since $\psi(f_d)+\psi(f')
+\psi(f'')=\psi(f)=x_i$ belongs to $V$, 
and $V$ is graded by assumption. 
By the minimality of $d$, 
we have $f'\in V$. 
Since $\psi=\phi \circ \eta $ belongs to $V^n$, 
it follows that 
$\psi(f')$ belongs to $V$. 
This implies that $h'$ belongs to $V$ as before. 
As for $h$ and $h''$, 
we have $h=f_d$ and $h''=0$, 
since $\psi(x_i)\in x_i+
\bigoplus _{l\geqslant 2}\Rx _l$ 
holds for $i=1,\ldots ,n$. 
Thus, 
$h+h'+h''=f_d+h'$ does not belong to $V$, 
a contradiction. 
Therefore, $\psi^{-1}$ belongs to $V^n$. 
\end{proof}

We remark that, 
if $V=\sum _{s\in \Sigma }As$ 
for some $\Sigma \subset \Rx $ 
and an $R$-subalgebra $A$ of $\Rx $, 
and if $\phi (A)\subset A$ 
and $\phi (\Sigma )\subset V$ 
hold for each $\phi \in V^n$, 
then $V^n$ is closed under composition. 
For example, 
assume that $R$ has prime characteristic $p$, 
and define $R[\x ^{p^l}]:=R[x_1^{p^l},\ldots ,x_n^{p^l}]$ 
for each $l\ge 0$. 
Set $\Z _{d}:=\{ a\in \Z \mid a\ge d\} $ 
for each $d\in \Z$. 
Let $d,e\in \Z _0$ and $\emptyset \ne N\subset \Z _0$ 
be such that $d\le e$, 
and each $u,v\in N$ satisfy $u+v\in N\cup \Z _d$. 
Then, 
we define a graded $R$-submodule of $\Rx $ by 
$$
V:=V(d,e,N):=R[\x ^{p^d}]
+\sum _{i=1}^n\sum _{u\in N}R[\x ^{p^e}]x_i^{p^u}. 
$$
Let us prove that 
$V^n$ is closed under composition 
using the remark for $A:=R[\x ^{p^e}]$ and 
$\Sigma :=R[\x ^{p^d}]\cup 
\{ x_i^{p^u}\mid 1\le i\le n,u\in N\} $. 
First, note that 
$f\in V$ implies $f^{p^v}\in V$ 
for any $v\in N$, 
since each $u,v\in N$ satisfy 
$(x_i^{p^u})^{p^v}\in \{ x_i^{p^w}\mid w\in N\} \cup R[x_i^{p^d}]$, 
and $R[\x ^{p^d}]$ contains $R[\x ^{p^e}]$. 
Hence, $\phi (x_i^{p^u})$ belongs to $V$ 
for each $\phi \in V^n$, $1\le i\le n$ and $u\in N$. 
Clearly, 
$\phi (R[\x ^{p^l}])\subset R[\x ^{p^l}]$ holds 
for any $\phi \in \Rx ^n$ and $l\ge 0$. 
Therefore, 
we have $\phi (A)\subset A$ 
and $\phi (\Sigma )\subset V$ 
for each $\phi \in V^n$.

Now, 
let us prove Theorem~\ref{thm:ngg} when $p$ is a prime. 
Clearly, $u,v\in \zs $ implies $u+v\in \zs \cup \Z _1$. 
So, we define $V_n:=V(1,1,\zs )$ 
and $W_n:=V(1,2,\zs )$, 
i.e., 
\begin{align*}
V_n&=R[\x ^p]+R[\x ^p]x_1+\cdots +R[\x ^p]x_n\\
W_n&=R[\x ^p]+R[\x ^{p^2}]x_1+\cdots +R[\x ^{p^2}]x_n. 
\end{align*}
If $n=p=2$, 
then $t_1,t_2\ge 0$ satisfy $x_1^{t_1}x_2^{t_2}\in W_2$ 
if and only if 
$t_1,t_2\equiv 0\pmod{2}$, or 
$t_i\equiv 1$, $t_j\equiv 0\pmod{4}$ 
for some $i,j\in \{ 1,2\} $. 
Hence, 
we have $x_1^{t_1}x_2^{t_2}\not\in W_2$ 
if and only if $x_1^{t_1}x_2^{t_2}$ is a good monomial. 
Similarly, 
when $(n,p)\ne (2,2)$, 
a nonzero monomial $m$ is good if and only if 
$m$ does not belong to $V_n$. 
Thus, 
$\ngg _n(R)$ is equal to $W_2^2\cap \GA _2(R)$ if $n=p=2$, 
and to $V_n^n\cap \GA _n(R)$ otherwise. 
Therefore, 
$\ngg _n(R)$ is a subgroup of $\GA _n(R)$ 
by Lemma~\ref{lem:IP}. 
Note that 
$\langle \Aff _{n+r}(R),\ngg _{n}(R)\rangle $ 
is contained in $W_{n+r}^{n+r}$ if $n=p=2$, 
and in $V_{n+r}^{n+r}$ otherwise. 
Since $\E _n(R)\sm V_{n+r}^{n+r}$ 
contains $\ep (x_2x_3)$ if $n\ge 3$, 
and $\ep (x_2^2)$ if $p\ge 3$, 
while $\E _n(R)\sm W_{n+r}^{n+r}$ 
contains $\ep (x_2^{p+1})$ for any $n,p\ge 2$, 
we get the last part of 
Theorem~\ref{thm:ngg}. 
This completes the proof.

If $I$ is a proper ideal of $R$, 
then each $\phi \in \GA _{n+r}(R)$ induces an element 
$\phi _I$ of $\GA_{n+r}(R/I)$. 
Since 
$\E _n(R)\ni \tau \mapsto \tau _I\in \E _n(R/I)$ 
is surjective, 
$(\sigma \circ \tau )_I=\sigma _I\circ \tau _I$ 
for each $\sigma ,\tau \in \GA _{n+r}(R)$, 
and $\tau \in \Aff _{n+r}(R)$ implies $\tau _I\in \Aff _{n+r}(R/I)$, 
we see that 
\begin{equation}\label{eq:reduction}
\E _n(R)\subset \langle \Aff _{n+r}(R),\phi \rangle 
\ \ {\rm implies}\ \ 
\E _n(R/I)\subset \langle \Aff _{n+r}(R/I),
\phi _I\rangle 
\end{equation}
for any $r\ge 0$ and $\phi \in \GA _n(R)$. 
Hence, 
Theorem~\ref{thm:ngg} implies the following corollary.

\begin{cor}\label{cor:reduction}
Let $I$ be a proper ideal of $R$ 
such that $\ch (R/I)$ is zero or a prime. 
If $\phi \in \GA _n(R)$ satisfies $\phi _I\in \ngg _n(R/I)$, 
then 
$\E _n(R)\not\subset \langle \Aff _{n+r}(R),\phi \rangle $ 
holds for any $r\ge 1$. 
Here, 
$\ch (R/I)$ is the characteristic of $R/I$. 
\end{cor}

Now, we can prove Theorem~\ref{thm:NG}. 
By assumption, 
$I:=I_{\phi }$ is a proper ideal of $R$. 
If $p$ is a prime, 
then $R$ contains $\F _p$, 
and so $\ch (R/I)=p$. 
Hence, 
$\phi _I$ belongs to $\ngg _n(R/I)$ by the definition of $I$. 
Thus, 
$\phi $ is not stably co-tame 
by Corollary~\ref{cor:reduction}. 
If $p=0$, 
then $\phi _I$ is affine. 
By (\ref{eq:reduction}), 
this implies that $\phi $ is not stably co-tame.

Finally, 
we remark that, if $p=2$, then $R[x_2]$ is contained in $V_n$. 
Hence, we have $\E _2(R)\subset V_2^2$, 
and so $\T _2(R)\subset V_2^2$. 
Thus, 
for any $\phi \in \T _2(R)$, 
no monomial $x_1^{t_1}x_2^{t_2}$ 
with $t_1$ and $t_2$ odd appears in $\phi (x_1)$ and $\phi (x_2)$.  
If $R$ is a domain, 
the same holds for any $\phi \in \GA _2(R)$, 
since $\T _2(K)=\GA _2(K)$ for any field $K$ 
(cf.~\cite{Jung}, \cite{Kulk}).

\begin{lem}\label{lem:nilp}
Assume that $p=2$, 
$\phi \in \GA _2(R)$, $1\le i\le 2$ 
and $t_1,t_2\ge 0$ are odd. 
If $a\in R$ is the coefficient of $x_1^{t_1}x_2^{t_2}$ 
in $\phi (x_i)$, 
then $a$ belongs to the nilradical of $R$. 
\end{lem}
\begin{proof}
Let $\mathfrak{p}$ be any prime ideal of $R$. 
Since $R/\mathfrak{p}$ is a domain of characteristic two, 
$x_1^{t_1}x_2^{t_2}$ does not appear in 
$\phi _{\mathfrak{p}}(x_i)$ as mentioned. 
Hence, 
$a$ belongs to $\mathfrak{p}$. 
\end{proof}

\section{Proof of Theorem~\ref{thm:main2} (i)}\label{sect:main2}
\setcounter{equation}{0}

Assume that $R$ is a commutative $k$-algebra. 
To prove Theorem~\ref{thm:main2} (i), 
we verify that one of (a), (b) and (c) 
of Theorem~\ref{thm:main1} holds when $J_{\phi }=R$. 
If $J_{\phi }=R$, 
a good monomial appears in a polynomial satisfying 
the degree condition. 
This implies that $\# k>2$, 
and so there always exists $\xi \in R^*$ 
satisfying $\xi +1\in R^*$.

Take any $f(x)\in R[x]$ and write 
$f(x)=\sum _{i=0}^du_ix^{ip^e}$ 
with $d,e\ge 0$ and $u_i\in R$, 
where we regard $p^e=1$ if $p=0$. 
If $k$ contains $d+1$ distinct elements 
$\xi _0,\ldots ,\xi _d$, 
then $u_ix^{ip^e}$ 
can be written as a $k$-linear combination of 
$f(\xi _0x),\ldots ,f(\xi _dx)$ 
for each $0\le i\le d$ by linear algebra. 
This remark is used to prove the following lemma.

\begin{lem}\label{lem:vdm2}
If $f\in \Rx $ satisfies the degree condition, 
each monomial appearing in $f$ 
is written as a $k$-linear combination of 
$f(\xi _1x_1,\ldots ,\xi _nx_n)$ 
for $\xi _1,\ldots ,\xi _n\in k^*$. 
\end{lem}
\begin{proof}
We prove the lemma by induction on $n$. 
The case $n=0$ is clear. 
Assume that $n\ge 1$, 
and write $f=\sum _{i=0}^df_ix_1^{ip^e}$, 
where $d:=\# k-2$, $f_i\in \Rxh $ and $e\ge 0$. 
Take any monomial $m$ appearing in $f$. 
Then, 
there exists $i$ such that 
$m$ appears in $f_ix_1^i$. 
Since $\# k=d+2$, 
we may find distinct $\alpha _0,\ldots ,\alpha _d\in k^*$. 
Then, 
$f_ix_1^i$ is written as a $k$-linear combination of 
$f(\alpha _lx_1,x_2,\ldots ,x_n)$ 
for $l=0,\ldots ,d$ as remarked. 
By the choice of $i$, 
we have $m=m'x_1^i$ for some monomial $m'$ 
appearing in $f_i$. 
Note that 
$\deg _{x_j}^sf_i\le \deg _{x_j}^sf\le d$ for all $j$. 
Hence, 
by induction assumption, 
$m'$ is written as 
a $k$-linear combination of 
$f_i(\beta _2x_2,\ldots ,\beta _nx_n)$ 
for $\beta _2,\ldots ,\beta _n\in k^*$. 
Thus, 
$m=m'x_1^i$ is written as a $k$-linear combination of 
$f_i(\beta _2x_2,\ldots ,\beta _nx_n)x_1^i$'s, 
and therefore that of 
$f(\alpha _lx_1,\beta _2x_2,\ldots ,\beta _nx_n)$ 
for $l=0,\ldots ,d$ and 
$\beta _2,\ldots ,\beta _n\in k^*$. 
\end{proof}

Next, 
we define $\lambda \in \Aff _n(R)$ by $\lambda (x_i)=x_i+1$ 
for $i=1,\ldots ,n$.

\begin{lem}\label{lem:good eta}
If $x_1^{t_1}\cdots x_n^{t_n}$ is a good monomial, 
there appears in $\lambda (x_1^{t_1}\cdots x_n^{t_n})$ 
the monomial $ux_ix_j$ with $u\in k^*$ 
in the cases {\rm (I)}, {\rm (II)} and {\rm (III)}, 
$x_ix_j^2$ in the case {\rm (IV)}, 
and $x_i^3$ in the case {\rm (V)}, 
where $i,j\in \{ 1,\ldots ,n\} $, 
and where $i\ne j$ if {\rm (II)} or {\rm (IV)}. 
\end{lem}
\begin{proof}
In the case (IV), 
we have 
$\lambda (x_1^{t_1}x_2^{t_2})
=(x_i+1)(x_j^2+1)(x_i^4+1)^{t_i'}(x_j^4+1)^{t_j'}$ 
for some $i,j\in \{ 1,2\} $ with $i\ne j$, and $t_1',t_2'\ge 0$. 
Then, 
it is easy to see that the monomial $x_ix_j^2$ appears 
in $\lambda (x_1^{t_1}x_2^{t_2})$. 
Other cases are verified similarly. 
\end{proof}

Now, 
let us prove Theorem~\ref{thm:main2} (i). 
Since $J_{\phi }=R$ by assumption, 
we may find $a_1,\ldots ,a_r\in R$ 
and $g_1,\ldots ,g_r\in \sum _{i=1}^nR\phi (x_i)+\sum _{i=1}^nRx_i$ 
with $r\ge 1$ as follows:

\smallskip

\nd 
(1) For each $1\le l\le r$, 
a good monomial $m_l$ appears 
in $g_l$ with coefficient $a_l$. \\
(2) $\sum _{l=1}^ra_l=1$, 
and $g_1,\ldots ,g_r$ satisfy the degree condition.

\smallskip

Since $g_l$ satisfies the degree condition for each $l$, 
so does $m_l$, 
and hence so does $\lambda (m_l)$. 
Clearly, 
$g_l$ belongs to $M_{\phi }$. 
Hence, 
$m_l$ belongs to $M_{\phi }$ 
by Lemma~\ref{lem:vdm2} and (\ref{eq:M phi rmk}). 
Then, 
$\lambda (m_l)$ belongs to $M_{\phi }$, 
so each monomial appearing in $\lambda (m_l)$ 
belongs to $M_{\phi }$ 
similarly. 
By Lemma~\ref{lem:good eta}, 
there appears in $\lambda (m_l)$ 
the monomial $n_l:=u_la_lx_{i_l}x_{j_l}$ with 
$u_l\in k^*$ in the cases (I), (II) and (III), 
$n_l:=a_lx_{i_l}x_{j_l}^2$ in the case (IV), 
and $n_l:=a_lx_{i_l}^3$ in the case (V), 
where $i_l,j_l\in \{ 1,\ldots ,n\} $, 
and where $i_l\ne j_l$ if (II) or (IV). 
Since $n_l$ belongs to $M_{\phi }$, 
we have $a_lx_1x_2\in M_{\phi }$ in the case (II). 
The proof of Lemma~\ref{lem:key1} shows that 
$a_lx_1x_2^2\in M_{\phi }$ in the case (V) 
as well as in the case (IV).

If $n\ge 3$ and $p=2$, 
then $m_l$ must be of type (II) for all $l$. 
Hence, 
$M_{\phi }$ contains $\sum _{l=1}^ra_lx_1x_2=x_1x_2$. 
Therefore, 
(a) of Theorem~\ref{thm:main1} holds.

If $n=p=2$, 
then $m_l$ is of type (II) or (IV) or (V) for each $l$. 
In the case (II), 
$a_l$ is nilpotent by Lemma~\ref{lem:nilp}. 
Since $\sum _{l=1}^ra_l=1$, 
it follows that $u:=\sum 'a_l$ is a unit of $R$, 
where the sum $\sum '$ is taken over 
$l$ such that $m_l$ is of type (IV) or (V). 
Hence, 
$M_{\phi }$ contains $u^{-1}\sum 'a_lx_1x_2^2=x_1x_2^2$. 
Therefore, 
(c) of Theorem~\ref{thm:main1} holds.

If $p\ne 2$, 
then each $m_l$ is of type (I) or (II) or (III), 
so $M_{\phi }$ contains $u_la_lx_{i_l}x_{j_l}$. 
By (\ref{eq:M phi rmk}), 
$M_{\phi }$ contains $a_lx_1^2$ if $i_l=j_l$. 
If $i_l\ne j_l$, 
then $M_{\phi }$ contains $a_lx_1x_2$, 
and hence contains $a_lx_1(x_2+x_1)-a_lx_1x_2=a_lx_1^2$. 
Thus, 
$M_{\phi }$ contains 
$\sum _{l=1}^ra_lx_1^2=x_1^2$. 
Therefore, 
(b) of Theorem~\ref{thm:main1} holds. 
This completes the proof of Theorem~\ref{thm:main2} (i).

\section{Remarks}\label{sect:rmk}
\setcounter{equation}{0}

\nd{\bf 6.1.} 
When $R=k$, 
Edo-Lewis~\cite{EL} showed that 
$\theta _N:=(\beta \circ \pi )^{-N}\circ \pi 
\circ (\beta \circ \pi )^N$ 
is neither affine nor co-tame for each $N\ge 3$, 
where 
$$
\beta :=(x_1+x_2^2(x_2+x_3^2)^2,
x_2+x_3^2,x_3)\quad{\rm and}\quad 
\pi :=(x_2,x_1,x_3). 
$$ 
Here, 
we mention that 
the expression of $\theta _N$ above 
is slightly different from the original one, 
due to the difference 
in the definitions of composition 
(cf.~\S 2.2 of \cite{EL} and (\ref{eq:composition})). 
Since $R=k$ and $\theta _N$ is not affine, 
$\theta _N$ is stably co-tame if $p=0$ 
by the remark after Corollary~\ref{cor:main}. 
If $p=2$, 
then $\theta _N$ belongs to $\ngg _3(k)$, 
and hence is not stably co-tame by Theorem~\ref{thm:ngg}. 
When $p\ge 3$, 
we have the following theorem.

\begin{thm}\label{thm:EL}
If $p\ge 3$, $N\ge 1$ and $\# k\ge 4^{2N-1}+2$, 
then $\theta _N$ is stably co-tame. 
\end{thm}
\begin{proof}
Observe that $\theta _N$ is written as 
$\sigma _1\circ \tau _1\circ \cdots \circ 
\sigma _N\circ \tau _N\circ \sigma $, 
where 
$$
\sigma _i\in \{ \pi \circ \beta \circ \pi , 
\pi \circ \beta ^{-1}\circ \pi \} ,\quad 
\tau _i\in \{ \beta ,\beta ^{-1}\} \quad 
{\rm and}\quad 
\sigma \in \{ \id ,\pi \}
$$
with 
$\sigma _1=\pi \circ \beta ^{-1}\circ \pi 
=:(h_1,h_2,x_3)$. 
For each $\w =(w_1,w_2,w_3)\in \Z ^3$, 
we define the $\w $-weighted grading on $\Rx $ 
by $\degw x_i:=w_i$ for $i=1,2,3$. 
For $f\in \Rx \sm \zs $, 
we denote by $f^{\w }$ the highest $\w $-homogeneous part of $f$. 
Now, set $\theta _N':=\theta _N\circ \sigma $. 
We claim that, 
if $\degw h_2$ is greater than $(1/4)\degw h_1$ and $2w_3$, 
then 
$$
\theta _N'(x_i)^{\w }
=(\sigma _1\circ \tau _1\circ \cdots \circ \sigma _N\circ \tau _N)(x_i)^{\w }
=\pm (h_2^{\w })^{4^{2N-i}}{\rm \ for\ }i=1,2. 
$$
In fact, 
the case $N=1$ is checked directly, 
and the case $N\ge 2$ follows by induction on $N$, 
since 
$(\sigma _i\circ \tau _i)(x_2)=x_2\pm x_1^2(x_1\pm x_3^2)^2\pm x_3^2$ 
and 
$$
(\sigma _i\circ \tau _i)(x_1)=
x_1\pm x_3^2\pm (x_2\pm x_1^2(x_1\pm x_3^2)^2)^2
(x_2\pm x_1^2(x_1\pm x_3^2)^2\pm x_3^2)^2
\quad{\rm for\ each\ }i. 
$$
Since 
$h_1=x_1-x_3^2$ and $h_2=x_2-x_1^2h_1^2$, 
the assumption of the claim holds for 
$\w =(1,0,0),(0,1,0),(0,0,1),(2,0,1)$, 
for which 
$$
h_2^{\w }=-x_1^4,x_2,-x_1^2x_3^4,-x_1^2(x_1-x_3^2)^2, 
$$ 
respectively. 
The first three of these imply 
$\deg _{x_i}\theta _N'(x_2)\le 4\cdot 4^{2N-2}\le \# k-2$ 
for $i=1,2,3$. 
Hence, 
$\theta _N'(x_2)$ satisfies the degree condition. 
Since $\theta _N'(x_2)^{(2,0,1)}=\pm 
(x_1(x_1-x_3^2))^t$ with $t:=2\cdot 4^{2N-2}$ 
similarly, 
the monomial 
$\pm (x_1x_3^2)^t$ appears in $\theta _N'(x_2)$. 
Since $p\ge 3$, 
we see that $x_1^tx_3^{2t}$ is a good monomial of type (III). 
Thus, we get $J_{\theta _N}=J_{\theta _N'}=R$. 
Therefore, 
$\theta _N$ is stably co-tame by 
Theorem~\ref{thm:main2} (i). 
\end{proof}

\nd{\bf 6.2.} 
When $k$ is a finite field, 
Theorem~\ref{thm:main2} might not be useful due to 
the degree condition. 
In such a case, 
the following technique may be effective. 
Take any $\eta \in \Aff _n(R)$, 
and set $\delta :=\eta -\id $. 
Then, 
$\delta (M_{\phi })\subset M_{\phi }$ 
holds for each $\phi \in \GA _n(R)$ 
by (\ref{eq:M phi rmk}). 
Since $\phi (x_j)$ belongs to $M_{\phi }$ 
for each $j$, 
we get the following lemma.

\begin{lem}\label{lem:delta}
If $\phi \in \GA _n(R)$ and $h\in \Rx $ 
satisfy the following condition, 
then $h$ belongs to $M_{\phi }$$:$ 
{\rm There exist $1\le j\le n$  and 
$\eta _1,\ldots ,\eta _l\in \Aff _n(R)$ with $l\ge 1$ 
such that 
$(\delta _1\circ \cdots \circ \delta _l\circ \phi )(x_j)$ 
belongs to $R^*h+M_{\phi }$, 
where $\delta _i:=\eta _i-\id $ for 
each $i$.} 
\end{lem}

Assume that $p$ is a prime. 
For $i=1,\ldots ,n$, 
take any $a_i\in R^*$, 
define $\eta _i\in \Aff _n(R)$ 
by $\eta _i(x_i)=x_i+a_i$ 
and $\eta _i(x_j)=x_j$ for $j\ne i$, 
and put $\delta _i:=\eta _i-\id $. 
Then, 
we have $\delta _i^p=0$ since $\eta _i^p=\id $, 
$\deg _{x_i}\delta _i(x_i^d)<d$ for $d\ge 0$, 
and $\delta _i^t(x_i^d)\in R^*x_i^{d-t}
+\sum _{j=0}^{d-t-1}Rx_i^j$ 
if $p\nmid d-l$ for all $0\le l<t$. 
Note that $\ker \delta _i$ is an $R$-subalgebra of $\Rx $, 
and $\delta _i$ 
is $\ker \delta _i$-linear. 
Also, 
$\ker \delta _i$ contains the $R$-subalgebra $A_i$ 
generated by $q_i:=x_i^p-a_i^{p-1}x_i$ and $x_j$ for $j\ne i$. 
We claim that $\ker \delta _i=A_i$. 
In fact, 
each $f\in \ker \delta _i\sm \zs $ 
satisfies $p\mid \deg _{x_i}f=:lp$,  
hence $\deg _{x_i}(f-hq_i^l)<lp$ holds 
for some $h\in R[\{ x_j\mid j\ne i\} ]$. 
Thus, 
we get $f-hq_i^l\in A_i$ by induction on $l$, 
and so $f\in A_i$.

Now, 
fix $l=(l_1,\ldots ,l_n)\in \{ 0,\ldots ,p-1\} ^n$, 
and set 
$\delta ^l:=\delta _1^{l_1}\circ \cdots \circ \delta _n^{l_n}$ 
and 
$$
M_l:=\sum _{i=0}^nRx_i\x ^l+\sum _{i=1}^n\sum _{j=0}^{l_i-1}A_ix_i^j, 
\ {\rm where}\ 
x_0:=1 \ {\rm and} \ \x ^l:=x_1^{l_1}\cdots x_n^{l_n}. 
$$
Then, 
since $\delta _i$'s commute with each other, 
and are $A_i$-linear, 
we see that $\delta ^l(M_l)$ 
is contained in $\sum _{i=0}^nRx_i$. 
Hence, 
$\delta ^l(M_l)\subset M_{\phi }$ 
holds for every $\phi \in \GA _n(R)$. 
Using this, 
we can give sufficient conditions 
for stably co-tameness of $\phi $. 
For example, assume that 
$l_1,l_2\le p-2$. 
Then, 
$\delta ^l(x_1x_2\x ^l)$ 
belongs to $R^*x_1x_2+\sum _{i=0}^2Rx_i$. 
Hence, 
if $\phi $ satisfies $\phi (x_j)\in R^*x_1x_2\x ^l+M_l$ 
for some $1\le j\le n$, 
then 
$(\delta ^l\circ \phi )(x_j)$ 
belongs to $R^*x_1x_2+M_{\phi }$. 
This implies that $x_1x_2$ belongs to $M_{\phi }$ 
by Lemma~\ref{lem:delta}. 
Therefore, 
$\phi $ is stably co-tame by Theorem~\ref{thm:main1}. 
Similarly, 
if $p\ge 3$, $l_1\le p-3$, 
and $\phi (x_j)$ belongs to $R^*x_1^2\x ^l+M_l$ 
for some $1\le j\le n$, 
then $\phi $ is stably co-tame.

\smallskip

\nd{\bf 6.3.} 
When $n\ge 3$, 
Motoki Kuroda~\cite{MK} 
showed that $\ep (x_2^2)$ is stably co-tame for $R=\F _p$ with $p\ge 3$, 
and found it hard to decide whether 
$\ep (x_2^5)$ is stably co-tame for $R=\F _3$. 
Note that $\ep (x_2^5)$ is stably co-tame for $R=\F _9$ 
by Theorem~\ref{thm:main2} (ii).

\smallskip

\nd{\bf 6.4.} 
The {\it Tame Generators Problem} asks if $\GA _n(R)=\T _n(R)$.  
In the context of stably co-tame automorphisms, 
$\ngg _n(R)$ is a natural generalization of $\Aff _n(R)$. 
So, 
the following problem is of interest, 
as a generalization of the Tame Generators Problem 
in prime characteristic 
(see also~\cite{EK} for other generalizations).

\smallskip 

\nd {\bf Problem.}\! 
When $p$ is a prime, 
does it hold that $\GA _n(R)=\langle \ngg _n(R),\T _n(R)\rangle $?

\noindent
Department of Mathematics and Information Sciences\\ 
Tokyo Metropolitan University \\
1-1  Minami-Osawa, Hachioji \\
Tokyo 192-0397, Japan\\
kuroda@tmu.ac.jp


\begin{thebibliography}{ASM}

\bibitem[BEW]{BEW}
J. Berson, A. van den Essen\ and\ D. Wright, 
{\it 
Stable tameness of two-dimensional polynomial automorphisms 
over a regular ring}, 
Adv. Math. {\bf 230} (2012), 
no.~4-6, 2176--2197. 


\bibitem[B]{B2}
Y. Bodnarchuk, 
{\it 
On generators of the tame invertible polynomial maps group}, 
Internat. J. Algebra Comput. {\bf 15} (2005), no.~5-6, 851--867. 


\bibitem[Ed]{Edo cotame}
E. Edo, 
{\it 
Coordinates of $R[x,y]$$:$ 
constructions and classifications}, 
Comm. Algebra {\bf 41} (2013), no.~12, 4694--4710. 


\bibitem[EK]{EK}
E. Edo\ and\ S. Kuroda, 
{\it 
Generalisations of the tame automorphisms over 
a domain of positive characteristic}, 
Transform. Groups {\bf 20} (2015), no.~1, 65--81. 


\bibitem[EL]{EL}
E. Edo\ and\ D. Lewis, 
{\it 
The affine automorphism group of $\Bbb A\sp 3$ 
is not a maximal subgroup of the tame automorphism group}, 
Michigan Math. J. {\bf 64} (2015), no.~3, 555--568. 



\bibitem[Es]{Essen}A.~van den Essen, 
{\it 
Polynomial automorphisms and the Jacobian conjecture}, 
Progress in Mathematics, 
Vol.\ 190, Birkh\"auser, Basel, Boston, Berlin, 
2000. 



\bibitem[J]{Jung}H.~Jung, 
{\it 
\"Uber ganze birationale Transformationen der Ebene}, 
J.\ Reine Angew.\ Math.\ {\bf 184} (1942), 161--174. 


\bibitem[K]{Kulk}W.~van der Kulk, 
{\it 
On polynomial rings in two variables}, 
Nieuw Arch.\ Wisk. (3) {\bf 1} (1953), 33--41. 


\bibitem[MK]{MK}M.~Kuroda, 
{\it Stably Derksen polynomial automorphisms over finite fields} 
(Japanese), 
Master's Thesis, 
Tokyo Metropolitan University, 
January 2016. 


\bibitem[MW]{MW}
S. Maubach\ and\ R. Willems, 
{\it 
Polynomial automorphisms over finite fields: 
mimicking tame maps by the Derksen group}, 
Serdica Math. J. {\bf 37} (2011), 
no.~4, 305--322 (2012). 


\bibitem[N]{Nagata}M.~Nagata, 
{\it 
On Automorphism Group of $k[x,y]$}, 
Lectures in Mathematics, Department of Mathematics, 
Kyoto University, Vol.\ 5, 
Kinokuniya Book-Store Co.\ Ltd., Tokyo, 1972. 


\bibitem[SU]{SU}
I.~Shestakov and U.~Umirbaev, 
{\it 
The tame and the wild automorphisms of polynomial rings 
in three variables}, 
J.\ Amer.\ Math.\ Soc.\ {\bf 17} (2004), 197--227. 



\bibitem[S]{Smith}M. K. Smith, 
{\it 
Stably tame automorphisms}, 
J. Pure Appl. Algebra {\bf 58} (1989), 
209--212. 



\end{thebibliography}
\end{document}